\documentclass[graybox]{svmult}

\usepackage{mathptmx}     
\usepackage{helvet}   
\usepackage{courier}
\usepackage{amsfonts} 
\usepackage{amssymb}
\usepackage{makeidx}         
\usepackage{graphicx}        
\usepackage{multicol}        
\usepackage[bottom]{footmisc}

\makeindex             

\begin{document}

\title*{Semi-magic matrices for dihedral groups}
\author{Robert W. Donley, Jr.}
\institute{Robert W. Donley, Jr., \at Queensborough Community College, Bayside, New York, \email{RDonley@qcc.cuny.edu}}
%
%
\maketitle

\abstract*{After reviewing the group structure and representation theory for the dihedral group $D_{2n},$ we consider an intertwining operator $\Phi_\rho$ from the group algebra $\mathbb{C}[D_{2n}]$ into a corresponding space of semi-magic matrices.  From this intertwining operator, one obtains the generating function for enumerating the associated semi-magic squares with fixed line sum and an algebra extending the circulant matrices.  While this work complements the approach to $D_{2n}$ through permutation polytopes, we use only methods from representation theory. }

\abstract{After reviewing the group structure and representation theory for the dihedral group $D_{2n},$ we consider an intertwining operator $\Phi_\rho$ from the group algebra $\mathbb{C}[D_{2n}]$ into a corresponding space of semi-magic matrices.  From this intertwining operator, one obtains the generating function for enumerating the associated semi-magic squares with fixed line sum and an algebra extending the circulant matrices.  While this work complements the approach to $D_{2n}$ through permutation polytopes, we use only methods from representation theory. }

\keywords{character, circulant matrix, dihedral group, group algebra, intertwining operator,  orthogonal idempotent, semi-magic square}

\section{Introduction}
\label{sec:1}

Just over a century ago, the problem of counting semi-magic squares was initiated by MacMahon \cite{PMM}, who gave formulas for enumerating such squares of size three with fixed line sum.  Problems of this type remain an active area of study in combinatorics and, in particular, the subject of permutation polytopes.    The theory involved  touches on topics including, but not exhaustively, permutation matrices (\cite{Bkf}), Ehrhart polynomials (see \cite{BSa} or \cite{StE} for general background), Stanley's proof of the Anand-Dumir-Gupta conjecture (\cite{ADG}, \cite{StC}), toric varieties, elliptic curves and zeta functions, and $3j$-symbols in the quantum theory of angular momentum, in the form of Regge symbols and Regge symmetries.

A natural source of semi-magic squares comes from permutation polytopes, which arise from realizations of finite groups as subgroups in a symmetric group $S_n,$ the permutations of which may in turn be realized as permutation matrices.  The permutation matrices are the fundamental atoms of the theory.  Because these constructions are rooted in group theory, we approach the subject with the tools of representation theory, with an emphasis on intertwining operators, in particular, the homomorphism to permutation matrices. See \cite{GP}, Theorem 3.2.  With this morphism of central interest, we place these sets of semi-magic squares in vector spaces, in fact, semisimple associative algebras,  and obtain both old and new results through consideration of  the kernel and image.  

In our case, we consider semi-magic squares associated with the dihedral groups $D_{2n}$ with $2n$ elements.  Combinatorial questions, such as face structure of the permutation polytope, are well-understood (\cite{HSt}, \cite{BHN}, \cite{BDO}), with several formulas given for the Ehrhart polynomials.  Although we obtain the same formulas here, our narrower concern is counting semi-magic squares obtained from a set of generators. We consider this question using an elementary model (\cite{StE}, p. 225, Problem 15; p. 561, Problem 53), give a simple form of the generating function, and reconcile some of the several formulas. 

On the other hand,  because semi-magic matrices are closed under matrix multiplication, the image of the permutation map generalizes the commutative algebra of circulant matrices (for instance, \cite{Da}, \cite{KS}), which has been a subject of traditional, ongoing interest.  Again we use methods of representation theory, in particular, characters and projection formulas, to describe this extension both in representation theoretic terms and in terms of the Artin-Wedderburn theorem for semisimple associative algebras.   We do not consider properties as a Lie algebra as found in \cite{BFF}, although we give bases with quaternionic properties as a further source of distinguished semi-magic matrices.

We note that, while most  results here require only the real numbers, it is helpful to assume all representations are over the complex numbers.  The first three sections of \cite{FH} cover all required background from representation theory. For a thorough account with consideration of general fields and for Artin-Wedderburn theory of semisimple associative algebras, see \cite{CR}.

\section{Definitions and notations}
\label{sec:2}

Fix $n\ge 3$, and let $G=D_{2n}$ denote the dihedral group with $2n$ elements. We consider four realizations of this group:

\begin{itemize}
\item the symmetry group of the regular $n$-gon,
\item a presentation on two generators with three relations,
\item a subgroup of the symmetric group on $n$ elements, and
\item a subgroup of the group of permutation matrices of size $n$.
\end{itemize}
It will be extremely convenient to switch between these realizations, depending on which properties we wish to emphasize.  Let $e$ denote the identity in the first three realizations.

In the first realization, we orient the regular $n$-gon  symmetrically about the origin in the $xy$-plane, labeling the vertices counter-clockwise by 1 through $n$, with 1 closest to the $x$-axis in the first quadrant.  Vertices lie on the $x$-axis only when $n$ is odd, in which case the vertex labeled $\frac{n+1}2$ lies on the negative $x$-axis.   

Denote by $R$ the counter-clockwise rotation about the origin by $\frac{2\pi}{n}$ and by $C$ the reflection across the $x$-axis (or complex conjugation).   Then $D_{2n}$ contains $n$ rotations of the form 
$$R^k\qquad (0\le k < n)$$
and $n$ reflections
$$CR^k\qquad (0\le k < n).$$
With $|x|$ denoting the order of the element $x$, the second realization may be given by the relations
\begin{equation}|R|=n,\qquad |C|=2,\qquad CRC=R^{-1}.\end{equation}
In cycle notation for permutations,
$$R=(12\dots n),\qquad C=(1n)(2\ n-1)\dots ,$$
and one readily verifies these permutations satisfy the three relations.

Finally, one may assign to each permutation $\sigma$ its corresponding permutation matrix $P_\sigma$:  for $1\le i, j\le n,$
$$if\quad \sigma(i)=j\quad then\quad [P_\sigma]_{ji}=1,  \quad otherwise \quad [P_\sigma]_{ji}=0.$$
If $\{e_i\}$ denotes the standard basis of $\mathbb{C}^n$, then $P_\sigma(e_i)=e_{\sigma(i)}$ and $P_\tau P_\sigma=P_{\tau\circ \sigma}$, where $\circ$ denotes composition of maps, usually omitted henceforth.
 The rotations belong to the space of  circulant matrices; that is, the corresponding matrices have all entries equal to zero except for one diagonal of ones, continuing through the left-hand side.  The diagonal of ones for $R^k$ starts in column one at row $k+1$.  On the other hand,  with our choice of numbering, the reflections also correspond to $(-1)$-circulant matrices, now with the diagonal of ones to the left.  In particular, the counteridentity matrix $P_C$ has a diagonal of ones along the main diagonal to the left. For example, when $n=4,$
$$P_R=P_{(1234)}=\left|
\begin{array}{cccc} 
 0\ & 0\ & 0\ & 1 \\
 1\ & 0\ & 0\ & 0 \\
 0\ & 1\ & 0\ & 0 \\
 0\ & 0\ & 1\ & 0
\end{array}\right|,\quad P_C=P_{(14)(23)}=\left|
\begin{array}{cccc} 
0\ & 0\ & 0\  & 1 \\
0\ & 0\ & 1\  & 0 \\
0\ & 1\ & 0\  & 0 \\
1\ & 0\ & 0\ & 0
\end{array}\right|.$$

Multiplication by $P_C$ on the left inverts columns, while  multiplication on the right inverts rows.  Thus multiplication by $C$ on either side switches between rotations and reflections or, by $P_C$,  circulant and $(-1)$-circulant elements of $D_{2n}$.

\section{Structure of $D_{2n}$}
\label{sec:3}

We review the basic structure of $D_{2n}$, with an emphasis on features needed for character tables.  In particular, we consider conjugacy classes and the commutator subgroup.  Both items vary based on the parity of $n$. An efficient calculation of both items follows directly from (1).  

Denote the conjugacy class of $\sigma$ by $C_\sigma,$ and recall that the commutator subgroup $[G, G]$ of $G$ is generated by the commutators
$$\{xyx^{-1}y^{-1}\ |\ x, y\in G\}.$$

\begin{example}[$n$ odd] There are $\frac{n+3}2$ conjugacy classes in $D_{2n}$, given by 
\begin{itemize}
\item $\{e\}$
\item $\frac{n-1}2$ classes of the form $\{R^{\pm k}\}$ for $1 \le k \le \frac{n-1}2$
\item one single class $C_C$ of size $n$ consisting of  all reflections.
\end{itemize}

The commutator subgroup consists of the rotation subgroup $\langle R\rangle$, so that the abelianization $D_{2n}/[D_{2n}, D_{2n}]$ of $D_{2n}$ has two elements.  Thus the character group  consists of two elements
$$D_{2n}^*=\{\chi_{triv},\   \chi_{det}\},$$
where, for all  $g$ in  $D_{2n},$ 
$$\chi_{triv}(g)=1,\qquad 
\chi_{det}(R^k)=1, \qquad \chi_{det}(CR^k)=-1.$$

With $n$ odd,  the $sgn$ character 
$$\chi_{sgn}(\sigma)=det(P_\sigma)$$ 
alternates between $\chi_{triv}\  (\frac{n-1}2\ even)$ and $\chi_{det}\  (\frac{n-1}2\ odd)$ based on the number of transpositions in $C=(1n)(2\ n-1)\dots$. 
\end{example}

\begin{example}[$n$ even]
When $n=2m$, there are $m+3$ conjugacy classes, given by 
\begin{itemize}
\item $\{e\}$
\item $r^m,$ the nontrivial central element,
\item $m-1$ classes of the form $\{R^{\pm k}\}$ for $1 \le k < m$, and
\item two reflection classes  $C_C=C_{CR^2}$ and $C_{CR}$, each of size $m$.
\end{itemize}

The splitting of the reflections is also seen in permutations; conjugation preserves cycle structure, and reflections have either no fixed points ($C$) or two fixed points ($CR$).

The commutator subgroup now consists of the subgroup $\langle R^2\rangle$, so that the abelianization of $D_{2n}$ has four elements, given by
$$\mathbb{Z}/2 \times \mathbb{Z}/2\ \cong\ \{{\bar e},\ {\bar R}, {\bar C},\ \overline{CR}\};$$
the nontrivial elements have order two. 
Thus the character group consists of four elements, given by
$$D_{2n}^*=\{\chi_{triv},\   \chi_{det},\ \chi_{sgn},\ \chi_{det}\cdot\chi_{sgn}\}.$$

As before, the values of $\chi_{sgn}$ and $\chi_{det}\cdot \chi_{sgn}$ on the reflection classes alternate based on the cycle structure of $C$.
\end{example}

\section{Irreducible types for $\rho$}
\label{sec:4}

Now we  define the permutation representation
$$\rho: D_{2n} \rightarrow GL(n, \mathbb{C}),$$
$$\rho(\sigma)=P_\sigma$$
and give its decomposition in terms of irreducible representations of $D_{2n}.$   As a permutation representation, the character of $\rho$ encodes the fixed points for each permutation.

Let $\pi_2: D_{2n} \rightarrow GL(2, \mathbb{C})$
be the two-dimensional representation defined by extending the action of $D_{2n}$ on $\mathbb{R}^2$ to $\mathbb{C}^2.$   One verifies from Tables 1 and 2 below that $\pi_2$ is irreducible.   

\begin{table}
\caption{Character table for $D_{2n}$ ($n$ odd)}
\label{tab:1}
\begin{tabular}{| c | c | c | c |}
\hline
\ \ $\qquad \sigma \qquad $  &  $\quad e\quad $  &  $\qquad R^{\pm k}\qquad $   & $\qquad C\qquad $  \\ [0.5ex]
\hline
$\quad |C_\sigma|\quad$ & $\ \ 1\ \ $ & $\ \ 2\ \ $  & $\ \ n\ \ $ \\ [0.5ex]
\hline
$\chi_{triv}$ & 1 & $1$    & $\ \ \ 1$  \\
$\chi_{det}$ & $1$ & $1$ & $-1$  \\
$\pi_2$  & $2$ & \ \  $2\cos(2k\pi/n)$ \ \  & $\ \ \ 0$ \\ [0.5ex]
\hline
$\rho$ & $n$ & $0$ &\ \ \  $1$\\
\hline
\end{tabular}
\end{table}

\begin{table}
\caption{Character table for $D_{2n}$ ($n=2m$)}
\label{tab:1}
{\small
\begin{tabular}{| c | c | c | c | c | c | c |}
\hline
$\qquad \sigma \qquad$  & $\quad e \quad$ & $\quad R^m\quad $ & $\quad R^{\pm k}$ \ ($k$  odd) $\quad$  & $\quad R^{\pm k}$ ($k$ even)$\quad$ & $\qquad C \qquad$  & $\qquad CR \qquad$  \\ [0.5ex]
\hline
$|C_\sigma| $ &$1$ &  $1$  & $2$   &  $2$  & $m$  & $m$   \\ [0.5ex]
\hline
$\chi_{triv}$ & 1 & $1$  & $1$ & $1$  & $1$ & $1$ \\
$\chi_{det}$ & $1$ & $1$ & $1$ & $1$ &$-1$ &  $-1$  \\
$\chi_{sgn}$ & 1 & $(-1)^m$  & $-1$ & $1$ & $(-1)^m\ \ \ $  & $(-1)^{m+1}$  \\
$\chi_{det}\cdot \chi_{sgn}$ & $1$ & $(-1)^m$ & $-1$ & $1$ & $(-1)^{m+1}$ &  $(-1)^m$\ \ \  \ \\
$\pi_2$  & $2$ &  $-2 $ & $2\cos(2k\pi /n)$ & $2\cos(2k\pi/n)$ & $0$  & $0$ \\
\hline
$\rho$ & $n$ & $0$ & $0$ & $0$  & $0$ & 2\\
\hline
\end{tabular}
}
\end{table}

For an efficient tabulation of all two-dimensional irreducible representations for $D_{2n}$, the mappings 
$$\varphi_{j}: D_{2n}\rightarrow D_{2n},$$
$$\varphi_{j}(R)=R^j,\qquad \varphi_{j}(C)=C$$
define homomorphisms for $0\le j < n.$    For $1\le j < n$, define 
$$\pi^j_2=\pi_2 \circ \varphi_{j}.$$
Then the complete set $\widehat{G}$ of irreducible classes is given by the one-dimensional characters and the two-dimensional representations $\pi_2^j$  (with $1\le j \le \frac{n-1}2$ for $n$ odd, and with $1 \le j < m-1$ for $n=2m$).  

In Table 1, the character table for $n$ odd, we append each character for $\pi^j_2,$ which equals that of $\pi_2$, but with the replacement $2\cos(\frac{2jk\pi}{n})$. Of course, since the characters agree,   $\pi_2^j$ and $\pi_2^{-j}$ are equivalent as representations.  

To see the row orthogonality relations, one uses basic trigonometric identities with the following lemma and the corresponding identity for sine.

\begin{lemma}
Suppose $n\ge 2$ and  $j$ is not a multiple of $n$. Then 
$$\sum\limits_{k=0}^{n-1}\ \cos(2jk\pi/n) = 0.$$
\end{lemma}
\begin{proof}
Suppose $\omega \ne 1$ is an $n$-th root of unity.  Then 
$$1+\omega+\omega^2+\dots+ \omega^{n-1}=0.$$
Taking the real part of each term with $\omega=e^{2j\pi i/n}$, the lemma follows.\qquad \qed
\end{proof}

Now orthogonality relations for characters immediately give

\begin{proposition}   Suppose $n$ is odd.  As  representations of $D_{2n}$,
$$\rho\ \cong\ \chi_{triv}\ \oplus\  \bigoplus\limits_{j=1}^{\frac{n-1}2}\  \pi_2^j.$$
\end{proposition}

Table 2 shows the character table when  $n=2m.$ Here the character for $\pi_2^j$ requires the change to  cosine as before and additionally the  $-2$ becomes $(-1)^j2$.   One also sees immediately that 
$$\pi_2^m\ \cong \chi_{sgn}\oplus \chi_{det}\cdot\chi_{sgn}.$$
\begin{proposition}   Suppose $n=2m.$  As  representations of $D_{2n}$,
$$\rho\ \cong\ \chi_{triv}\ \oplus\  \chi' \ \oplus \ \bigoplus\limits_{j=1}^{m-1}\  \pi_2^j,$$
where
$$\chi'=\chi_{sgn}\  (m\ odd)\quad or\quad   \chi'= \chi_{det}\cdot \chi_{sgn}\  (m \ even).$$
\end{proposition}

\section{The Intertwining operator $\Phi_\rho$}
\label{sec:5}

We recall the definition of the group algebra for a finite group $G$.

\begin{definition}  Let $G$ be a finite group, and suppose $g_1,\dots, g_n$ are the elements of $G$.  Define $\mathbb{C}[G]$ to be the vector space of linear combinations over the formal basis 
$\{e_{g_1},\dots, e_{g_n}\}.$
We define multiplication by extending the multiplication of basis elements on indices:
$$e_{g}\cdot e_{g'} = e_{g\cdot g'}.$$
Note that $\dim_\mathbb{C}\mathbb{C}[G]=|G|.$
\end{definition}

If $(\pi, V_\pi)$ is a representation of $G$, then one extends $\pi$ to a homomorphism of associative algebras by
$$\Phi_\pi:  \mathbb{C}[G] \rightarrow Hom_\mathbb{C}(V_\pi, V_\pi),$$
$$\Phi_\pi(\sum x_i e_{g_i}) = \sum x_i \pi(g_i).$$
Both the kernel and image of $\Phi_\pi$ are subrepresentations of $G\times G$ in the domain and range, respectively, and one has the isomorphism of algebras
$$Im(\Phi_\pi)\ \cong \ \mathbb{C}[G]/ Ker\ \Phi_\pi.$$  
Here the corresponding actions of $G\times G$ are given by extending
$$(L\otimes R)(g_1, g_2) e_x = e_{g_1xg_2^{-1}},$$
$$(Hom(\pi, \pi)(g_1, g_2))M = \pi(g_1)M\pi(g_2^{-1}).$$

Additionally, we have the analogue of the Peter-Weyl theorem for finite groups:
\begin{proposition}
As representations of $G\times G$, 
$$\mathbb{C}[G]\  \cong\    \bigoplus\limits_{\pi'}\  Hom_\mathbb{C}(V_{\pi'}, V_{\pi'}),$$
where the sum ranges over the set $\widehat{G}$ of irreducible classes of representations for $G$.  Each summand on the right-hand side is irreducible for $G\times G.$
\end{proposition}

Each irreducible subrepresentation for $G\times G$ corresponds to a simple two-sided ideal of the algebra structure, and the one-dimensional ideals in $\mathbb{C}[G]$ are directly obtained as follows: 

\begin{proposition}
Suppose $\chi:  G \rightarrow \mathbb{C}^*$ is a (one-dimensional) character of $G$.  Then 
$$|G| e_\chi = \sum\limits_{g\in G}  \chi(g^{-1})e_g$$
spans the subspace of $\mathbb{C}[G]$ correpsponding to $Hom_\mathbb{C}(\mathbb{C}_\chi, \mathbb{C}_\chi)$.
\end{proposition}

Next Schur's Lemma now implies that there exists a subset $X_\pi$ of the set  $\widehat{G}$ of irreducible classes such that
$$Ker(\Phi_\pi) =  \bigoplus\limits_{\pi'\in X_\pi} Hom_\mathbb{C}(V_{\pi'}, V_{\pi'}),$$
$$Im(\Phi_\pi) =  \bigoplus\limits_{\pi'\notin X_\pi} Hom_\mathbb{C}(V_{\pi'}, V_{\pi'}).$$

From Propositions 1 and 2, we immediately obtain for the permutation representation $\Phi_\rho$ for $D_{2n}$:

\begin{proposition}
(a) If $n$ is odd, then $X_\rho=\{\chi_{det}\}$. That is,
$$Ker(\Phi_\rho) = \mathbb{C}e_{\chi_{det}}.$$

(b)  If $n=2m$, then $X_\rho=\{\chi_{det}, \chi''\},$ where
$$\chi''=\chi_{sgn}\  (m\ even)\quad or \quad   \chi'' =\chi_{det}\cdot \chi_{sgn} \ (m \ odd).$$
That is,
$$Ker(\Phi_\rho) = \mathbb{C}e_{\chi_{det}} \oplus \mathbb{C}e_{\chi''}.$$
\end{proposition}

\section{Semi-magic Matrices}
\label{sec:6}

Closely related to permutation  representations are semi-magic matrices. We recall basic properties about these matrices, with attention to properties from the previous section.

\begin{definition}
An element $M$ of $M(n, \mathbb{C})$ is called a {\bf semi-magic matrix} if the sums along each row and column are equal.  This common sum $r_M$ is called the {\bf line sum} of $M$. Denote the vector space of semi-magic matrices of size $n$ by $MM(n).$
\end{definition}

\begin{example}
Every permutation matrix is a magic matrix with line sum 1, and in fact the set of all permutation matrices form a spanning set of $MM(n)$.  The line sum of a linear combination of permutation matrices is evident:
$$M=\sum x_\sigma P_\sigma\ \   \mapsto\ \  r_M=\sum x_\sigma.$$
\end{example}

\begin{example}
Consider the element $J$ in $MM(n)$ with all entries equal to one.  Then one sees immediately that the line sum of $J$ is $n$, and, for all $M$ in $MM(n)$ with line sum $r$,
$$JM=MJ=rJ,$$
so that $J$ is in the center of $MM(n)$.  In particular, $J^2=nJ.$
\end{example}
In general, 
\begin{proposition}  If $M_1$ and $M_2$ are in $MM(n)$ with line sums $r_1$ and $r_2$ then so is $M_1M_2$ with line sum $r_1r_2$.
\end{proposition}

\begin{proof}  Let $e=(1, \dots, 1)$ be in $\mathbb{C}^n.$ That $M$ has the magic property with line sum $r$ is equivalent to the eigenvector condition
$$Me=M^Te=re.$$
Then 
$$M_1M_2e=r_2M_1e=r_1r_2e$$
and
$$(M_1M_2)^Te=M_2^TM_1^Te=r_1M_2^Te=r_2r_1e.\qquad \qed$$
\end{proof}

Thus $MM(n)$ is an associative algebra over $\mathbb{C}$, and the line sum functional is a linear character with respect to the algebra structure.  Furthermore, the corresponding map 
$$\Phi: \mathbb{C}[S_n]\rightarrow MM(n)$$
is a surjective intertwining operator between representations of $S_n\times S_n$ as before, and it is not difficult to see that the decomposition into irreducible subrepresentations is given by 
$$MM(n)\ \cong \mathbb{C}J\ \oplus MM_0(n),$$
where 
$$MM_0(n)=\{M\in MM(n)\ |\ JM=0\}$$ is the simple two-sided ideal in $MM(n)$ where all elements have line sum equal to zero.  Thus, for the center of $MM(n),$ 
$$Z(MM(n)) =Span_\mathbb{C}(I_n, J),$$
where $I_n$ is the identity matrix of size $n$.

More useful for what follows, the weaker statement that products of linear combinations are closed follows from the closure property of group multiplication:
$$(\sum x_\sigma P_\sigma)(\sum y_\tau P_\tau)=\sum x_\sigma y_\tau P_{\sigma\tau}.$$
With this note, the next two definitions merely recast the permutation mapping from the previous section.

\begin{definition}
Suppose $G$ is a subgroup of $S_n.$  Denote by $MM(G)\subseteq MM(n)$ the associative algebra over $\mathbb{C}$ generated by the elements of $G$ as permutation matrices. We call $MM(G)$ the {\bf permutation algebra} associated to $G.$  This algebra depends on the embedding of $G$ in $S_n.$
\end{definition}

\begin{definition}  The intertwining operator
$$\Phi_\rho: \mathbb{C}[G] \rightarrow MM(G)\subset MM(n),$$
$$\sum x_\sigma e_\sigma\ \  \mapsto\ \  \sum x_\sigma P_\sigma$$
is a surjection between both $G\times G$ representations and associative algebras over $\mathbb{C}$.  
With respect to either structure,
$$MM(G)\ \cong \mathbb{C}[G]/Ker(\Phi_\rho).$$
Additionally, $\Phi_\rho$ carries coefficient sums to line sums.
\end{definition}

\section{Semi-magic Squares for $MM(D_{2n})$}
\label{sec:7}

We now consider the problem of enumerating certain types of semi-magic squares.

\begin{definition}  A semi-magic matrix $M$ is called a {\bf semi-magic square} if it has entries in the non-negative integers.  Alternatively, $M$ may be written as
$$M=\sum n_\sigma P_\sigma$$
where each $n_\sigma$ is a non-negative integer. In this case, the line sum of $M$ is $\sum n_\sigma$.
\end{definition}

\begin{definition}  Let $H_G(r)$ denote the number of semi-magic squares in the permutation algebra $MM(G)$ with line sum equal to $r$.
\end{definition}

Traditionally $H_n(r)$ is used to denote the number of all semi-magic squares of size $n$ with line sum equal to $r$.  Problems of this type go back to MacMahon (1916) for the case of size three.
We base our results for  $H_{D_{2n}}(r)$ on the model found in \cite{StE}, p. 225, Problem 15 for the case of $n=3.$

\begin{theorem}
Suppose $n$ is odd and  $H_{D_{2n}}(r)$ counts the number of semi-magic squares in $MM(D_{2n})$ with line sum $r$.  Then 
$$H_{D_{2n}}(r) = \left(\begin{array}c  r+2n-1 \\ 2n-1  \end{array}\right) - \left(\begin{array}c  r+n-1 \\ 2n-1\end{array}\right),$$
which has generating function
$$F_n(x)=\sum\limits_{r\ge 0}  H_{D_{2n}}(r) x^r = \frac{1-x^n}{(1-x)^{2n}}.$$
\end{theorem}

\begin{proof}
Since the permutation matrices form a spanning set for $MM(D_{2n})$, we may identify each linear combination in $MM(D_{2n})$ with a $2n$-tuple of complex numbers, ordered so that the rotation elements precede the reflection elements.  That is,
$$\sum x_i P_{\sigma_i} \mapsto (x_1, \dots, x_{2n}).$$
By Proposition 5, the non-uniqueness of this representation is captured by the single dependence relation from $\chi_{det}$:
$$\sum\limits_{\sigma} \chi_{det}(\sigma) P_\sigma = 0\qquad or\qquad \sum\limits_{k} P_{R^k} = \sum\limits_{k} P_{CR^k} \ \ (=J).$$ 
In terms of $2n$-tuples, this relation presents as
$$(x_1+1,\ \dots,\ x_n+1,\ x_{n+1},\ \dots,\ x_{2n}) = (x_1, \ \dots,\ x_n,\ x_{n+1}+1,\ \dots,\ x_{2n}+1).$$
Thus each semi-magic square with line sum $r$ is uniquely represented by a $2n$-tuple of non-negative integers such that
\begin{itemize}
\item the entries sum to $r,$ and
\item at least one value in the last $n$ entries is equal to zero.
\end{itemize}
The first term of $H_{D_{2n}}(r)$ counts the number of way to place $r$ balls in $2n$ boxes (weak compositions of $r$ into $2n$ parts);  to guarantee a vanishing entry, we discard $2n$-tuples in which the last $n$ entries are nonzero, or of the form
$$(x_1,\ \dots,\ x_n,\ 1+x_{n+1},\ \dots,\ 1+ x_{2n})$$
where the $x_i$ are nonnegative integers and 
$$\sum\limits_{i=1}^{2n} x_i = r-n.$$
Finally, with $s\ge 1,$ the generating function follows from the binomial series
$$\frac{1}{(1-x)^s}= \sum\limits_{r\ge 0}  \left(\begin{array}{c} r+s-1 \\ s-1 \end{array}\right)  x^r.\qquad \qed$$ 
\end{proof}

\begin{theorem}
Supoose $n=2m$ and  $H_{D_{2n}}(r)$ counts the number of semi-magic squares in $MM(D_{2n})$ with line sum $r$.  Then 
$$H_{D_{2n}}(r) = \left(\begin{array}{c} r+4m-1 \\ 2n-1  \end{array}\right)-2\left(\begin{array}{c} r+3m-1 \\ 2n-1\end{array}\right)  + \left(\begin{array}{c} r+2m-1 \\ 2n-1\end{array}\right),$$
which has generating function
$$F_n(x)=\sum\limits_{r\ge 0}  H_{D_{2n}}(r) x^ r = \frac{(1-x^m)^2}{(1-x)^{2n}}.$$
\end{theorem}

\begin{proof}
The approach is essentially the same as the previous theorem, except now Proposition 5 yields two relations
$$\sum \chi_{det}(\sigma) P_\sigma =0\qquad and\qquad \sum \chi''(\sigma) P_\sigma = 0,$$
which, noting the sign for $r^m$, reduce to the equalities 
$$\sum P_{R^{2k}} = \sum P_{CR^{2k+1}} \ \ (=J_1),\qquad  \sum P_{R^{2k+1}} = \sum P_{CR^{2k}}\ \ (=J_2).$$
Here $J_1$ is a matrix of alternating ones and zeros, starting with a one in the upper left corner;  a similar pattern holds for $J_2$ but begins with a zero in this corner.   

If we order the group elements as above with four sections, we are now counting as before but with at least one zero entry in both the second and fourth sections.  If $r=r_1+r_2$, then the distribution of $r_1$ to the first $n$ entries and $r_2$ to the last $n$ entries is independent, so that
$$H_{D_{2n}}(r) = \sum\limits_{r_1=0}^r H_{D_n}(r_1)H_{D_n}(r-r_1),$$
Noting that this sum is a discrete convolution, we obtain the generating function by squaring that of $H_{D_{n}}(r),$  from which we obtain the counting formula. \qquad \qed
\end{proof}
The equalities with $J_1$ and $J_2$ are directly seen using circulant matrices for rotations and reflections, recalling that multiplication by $P_C$ on the left inverts columns.   Also note that
$$J=J_1+J_2,\quad J_1^2=mJ_1, \quad J_2^2=mJ_1, \quad J_1J_2=J_2J_1=mJ_2.$$

Finally, the following corollary gives the analogue of the Anand-Dumir-Gupta conjecture for $D_{2n}$ after Stanley.  

\begin{corollary}  Suppose $G=D_{2n}\subset S_n$.  The following properties of $H_G(r)$ hold:
\begin{itemize}
\item for $n$ odd (resp. even), $H_G(r)$ is a polynomial of degree $2n-2$ (resp. $2n-3$),
\item $H_{G}(-r) = (-1)^{n-1}H_{G}(r-n),$
\item $H_{G}(-1)=H_{G}(-2)=\dots=H_{G}(-n+1)=0,$ and 
\item the coefficients of the  numerator $h^*(x)$ in the reduced generating function are positive, symmetric, and unimodal. 
\end{itemize}
\end{corollary}

\begin{proof} The counting formula shows that $H_G(r)$ is a polynomial in $r$, with degree given by reducing $F_n(x)$ and applying the binomial series.  Once reduced, the numerator is a polynomial in $x$ with coefficients as noted. The remaining properties follow by expanding the binomial coefficients above into factorials. \qquad \qed
\end{proof}

\section{Alternative Counting Formulas}
\label{sec:8}

Alternative formulas for $H_{D_{2n}}(r)$ arise by either reducing the generating function or changing counting method.  We consider when $n$ is odd; for the even case, a similar formula follows with alteration for the discrete convolution.  See \cite{BHN} for the $h^*$-vector and generating function in the context of permutation polytopes.  

\begin{corollary}
Suppose $n$ is odd, and $G=D_{2n}.$ Then 
$$H_G(r) = \sum\limits_{i=0}^{n-1} \left( \begin{array}{c} r+ 2n-2-i \\ 2n-2 \end{array}\right)$$
\end{corollary}
\begin{proof}
Using the proof of Theorem 1, we give an alternative counting scheme for the unique representations with sum $r$ as follows.
First, let $X_0$ denote all $2n$-tuples with a zero in the $2n$-th entry, and  let $X_i$ denote  all $2n$-tuples with the last $i$ entries  nonzero and a zero in the $2n-i$-th entry. That is, all $X_i$ have elements of the form
$$(0, 0, \dots, 0, 1, \dots, 1) + (x_1, x_2, \dots, x_{2n-i-1}, 0, x_{2n-i+1}, \dots),$$
where the first vector ends with $i$ ones and all $x_k \ge 0.$ Then, for $0\le i \le n-1$, the number of elements in $X_i$ is 
the $i$-th term of the sum.\qquad \qed
\end{proof}

Effectively, we have reduced the generating function and applied the binomial series.  When $n$ is odd, the numerator of the reduced generating function is 
$$h^*(x)= 1+x+x^2+\dots + x^{n-1};$$ when $n$ is even, this numerator is 
$$h^*(x) = (1+x+x^2+\dots+ x^{m-1})^2 = 1+2x+ 3x^2 + \dots + 3 x^{n-4} + 2 x^{n-3}+ x^{n-2}.$$

Another formula follows from the Principle of Inclusion-Exclusion (\cite{BSa}, p. 39); here we sum directly over the choice of $i$ positions for zeros in the last $n$ entries.  See Theorem 1.2 of \cite{BDO} for a similar formula.

\begin{corollary}
Suppose $n$ is odd, and $G=D_{2n}.$ Then 
$$H_G(r) = \sum\limits_{i=1}^{n}  (-1)^{i+1} \left(\begin{array}{c} n \\ i \end{array}\right)\left(\begin{array}{c} r+ 2n-1-i \\ 2n-1-i \end{array}\right).$$
\end{corollary}

To reconcile this formula with the generating function,  if we start the sum at $i=0$, then the binomial series gives
$$\sum\limits_{r\ge 0} H(r) x^r = \sum\limits_{i=0}^n (-1)^{i+1} \left(\begin{array}{c} n \\ i \end{array}\right) \frac{(1-x)^i}{(1-x)^{2n}} = \frac{-x^n}{(1-x)^{2n}}.$$
Now we subtract  the term at $i=0$ to get  the result.

\section{Orthogonal Idempotents for $MM(D_{2n})$}
\label{sec:9}

Let $N=\langle R\rangle$ be the normal subgroup of rotations in $D_{2n}.$  The algebra $MM(D_{2n})$ extends the commutative algebra $Circ(n)=MM(N)$ of circulant matrices of size $n$.  As a commutative algebra generated by  a permutation matrix, this algebra, of dimension $n$,  may be simultaneously diagonalized. In fact, with respect to the action of $N\times N$, we have
$$\mathbb{C}[N]\ \cong\ Im(\Phi_{\rho|_N})\ \cong \ \bigoplus\limits_{\chi\in N^*} Hom_\mathbb{C}(\mathbb{C}_\chi, \mathbb{C}_\chi),$$
where $N^*$ is the set of $n$ characters $(\chi, \mathbb{C}_\chi)$ on $N$. As a decomposition into simple two-sided ideals of $\mathbb{C}[N]$ and $Circ(n)$, the elements
$$e_\chi=\frac1{n}\sum \chi(R^{-k}) e_{R^k}\qquad and \qquad U_\chi = \Phi_\rho(e_\chi) = \frac1{n}\sum \chi(R^{-k}) P_{R^k}$$
respectively span the isotypic component $Hom(\chi, \chi)$ for $N\times N.$  In particular, if $\omega=e^{2\pi i/n}$, $0\le j \le n-1,$ and $\chi_j(R)=\omega^j,$  then
$$U_{\chi_j} = \quad \frac1{n}\ \sum\limits_{k=0}^{n-1}\ \ \omega^{-jk} \ P_{R^k} \quad = \quad  
\frac1{n}\ \left|\begin{array}{ccccc}  1 &\  \omega^j \ &\  \omega^{2j} &\ \  \omega^{3j} &\dots \\ \ \ \  \omega^{-j} & 1 & \omega^{j} & \  \ \omega^{2j} & \dots \\  \quad\,  \omega^{-2j} & \ \ \  \omega^{-j} & 1 &\  \omega^{j} & \dots \\ \ \dots & \dots & \dots & \dots &\dots  \end{array}\right|$$
is evidently circulant and semi-magic with line sum equal to one if $\chi$ is trivial, zero otherwise.

In addition to giving diagonalizing bases for these algebras, we obtain a complete set of orthogonal idempotents corresponding to each simple two-sided ideal; that is, for instance, on $Circ(n),$
$$U_\chi^2=U_\chi,\qquad U_\chi U_{\chi'}=0\ \ (\chi\ne \chi'), \quad \sum U_\chi = I.$$
Then, as a decomposition of simple two-sided ideals, 
$$Circ(n) = \sum\limits_{\chi} Circ(n)U_\chi = \sum\limits_{\chi} \mathbb{C}U_\chi.$$

As seen in Propositions 1 and  2, a similar decomposition for $MM(D_{2n})$ consists of simple two-sided ideals with diagonal blocks of size 1 and 2.  A full description of this decomposition is given by the corresponding complete set of orthogonal idempotents:

\begin{theorem} Suppose $n$ is odd, and define $c_j = \cos(2j\pi/n)$.  The complete set of $\frac{n+1}2$ orthogonal idempotents for $MM(D_{2n})$ is given by 
$$U_{triv} = \frac1{n}\ J,$$
and, for each $1\le j \le \frac{n-1}2$,
$$U_{\pi_2^j} = \quad \frac2{n}\ \sum\limits_{k=0}^{n-1}\ \cos(2kj\pi/n) \ P_{R^k} \quad = \quad  
\frac2{n}\ \left|\begin{array}{ccccc} \ \  1 & \ \ c_{j} & \ \ \ c_{2j} & \ \ c_{3j} &\dots \\  \ \ c_j & \ \ 1 & \ \ c_{j} & \ \ \ c_{2j} & \dots \\  \ \ \ c_{2j} & \ \ c_j & \ \ 1 & \ \  c_{j} & \dots \\  \dots & \dots & \dots & \dots &\dots  \end{array}\right|.$$
\end{theorem}

\begin{proof}
Propositions 1 and  2 identify the isotypic components for $D_{2n}$ in $\rho$.  To obtain each $U$, we note the general projection formula onto the isotypic component for $\pi$:
$$P_\pi(v) = \frac{d_\pi}{|G|} \sum\limits_{g} \ \chi_\pi(g^{-1}) \rho(g) v,$$
where $d_\pi$ is the dimension of the irreducible representation $(\pi, V_\pi).$
Since the blocks in the group algebra may be identified uniquely by the left or right action alone, we use the left action applied to the identity element to obtain
$$U_\pi = P_\pi(I) = \frac{d_\pi}{2n} \sum\limits_{\sigma} \chi_\pi(\sigma^{-1}) P_\sigma.$$
The idempotent, orthogonality, and completeness properties for $U$ are inherited from the $P_\pi.$\qquad \qed
\end{proof}

\begin{theorem} Suppose $n=2m$, and preserve the notation of the previous theorem.  The complete set of $m+1$ orthogonal idempotents for $MM(D_{2n})$ is given by 
$$U_{triv} = \frac1{n}\ J,$$
$$U_{\chi'} = \frac1{n}\ \left|\begin{array}{ccccc} \ \ \ 1 & -1 & \ \ \ 1 & -1 &\dots \\  -1 & \ \ \ 1 & -1 & \ \ \ 1 & \dots \\  \ \ \ 1 & -1 & \ \ \ 1 & -1 & \dots \\  \dots & \dots & \dots & \dots &\dots  \end{array}\right|$$
and,  as defined in Theorem 3,  each $U_{\pi_2^j}$ for $1\le j \le m-1$.
\end{theorem}

\begin{proof}
The proof is unchanged from the previous theorem.  For the second idempotent, it is helpful to note 
$$P_CJ_2 = P_C \sum P_{R^{2k+1}} = \sum P_{CR^{2k+1}} = \sum P_{R^{2k}}=J_1,$$ 
$$P_C J_1 = P_C \sum P_{R^{2k}} = \sum P_{CR^{2k}} = \sum P_{R^{2k+1}} = J_2.\quad \qed$$
\end{proof}

Several items are worth noting in both theorems:
\begin{itemize}
\item each $U$ is symmetric and circulant,
\item $U_{triv}$ has line sum equal to 1,
\item $U$ has line sum equal to 0 otherwise,
\item although true from general theory, that these $U$ form a complete set of orthogonal idempotents is directly  checked,
\item the set of all $U$ gives a basis for the center of $MM(D_{2n})$, and 
\item the two-dimensional idempotents are twice the real parts of the idempotents $U_{\chi_j}$ in $Circ(n)$. 
\end{itemize}

\begin{corollary}  As a sum of simple two-sided ideals,
$$MM(D_{2n})\   =\  \mathbb{C}J \  \oplus\   \bigoplus\limits_{j=1}^{\frac{n-1}2} MM(D_{2n}) U_{\pi_2^j}\quad (n\ {odd}),\ \ and$$
$$MM(D_{2n}) =  \mathbb{C}J\   \oplus\   \mathbb{C}U_{\chi'} \ \oplus\  \bigoplus\limits_{j=1}^{m-1} MM(D_{2n}) U_{\pi_2^j}\ \ (n=2m).$$
When $n$ is odd (resp. $n=2m$), the dimension of $MM(D_{2n})$ is $2n-1$ (resp. $2n-2$).
\end{corollary}

Now Propositions 1 and 2 are given explicitly by the following:

\begin{corollary}
Fix $n\ge 3$, define $c_j$ as before, and let $s_j=\sin(2j\pi/n)$.  The orthogonal decomposition of $\mathbb{C}^n$ into $D_{2n}$-invariant subspaces under $\rho$ is given by
$$\mathbb{C}^n = \mathbb{C}(1, 1, \dots, 1)\ \oplus\ \bigoplus\limits_{j=1}^{\frac{n-1}2}\ V_{\pi_2^j}\quad (n\ odd), \ \  and$$
$$\mathbb{C}^n = \mathbb{C}(1, 1, \dots, 1)\ \oplus\  \mathbb{C}(1, -1, 1, \dots, -1)\ \oplus\ \bigoplus\limits_{j=1}^{m-1}\ V_{\pi_2^j}\quad (n=2m).$$
Here $V_{\pi_2^j}= U_{\pi_2^j} \mathbb{C}^n$ with an orthonormal basis given by
$$u_j=\sqrt{\frac2{n}}\ (1, c_j, c_{2j}, \dots, c_{(n-1)j}),\qquad v_j=\sqrt{\frac2{n}}\ (0, s_j, s_{2j}, \dots, s_{(n-1)j}).$$
\end{corollary}

\begin{proof}
If
$$w_{j}=\frac1{\sqrt{n}}\ (1, \omega^j, \omega^{2j}, \dots, \omega^{(n-1)j}),$$
 a unit eigenvector for $R$ with eigenvalue $\omega^{-j},$ then the vectors
 $$u_j = \frac{w_j+w_{-j}}{\sqrt{2}},\qquad v_j = \frac{w_{j}-w_{-j}}{\sqrt{2}i}$$ 
 span a subspace invariant and irreducible under $D_{2n}$. That is, the span is invariant under all $\rho(R^k)$, and 
 $$\rho(RC)u_j = u_j, \quad \rho(RC)v_j = -v_j.$$
 Since $\rho(R)$ is an orthogonal matrix, these eigenvectors are orthogonal with respect to the usual Hermitian inner product on $\mathbb{C}^n,$ and orthonormality in the corollary follows immediately. 
 
 Alternatively, we may simply apply $U_{\pi_2^j}$ to the standard basis vectors $e_1$ and $e_2$ and use Gram-Schmidt orthogonalization, or we may unravel the multiplication of basis vectors in the next section. \qquad \qed
\end{proof}

\section{Quaternionic bases}
\label{sec:10}

Finally we use the results of the previous section to give bases for each component of type $\pi_2^j$ in Corollary 4.  In addition to obtaining further examples of special semi-magic matrices, we exhibit the quaternionic structure for such components. 

Now we consider the imaginary part of $U_{\chi_j}$.   With $s_j = \sin(2j\pi/n),$ we define
$$U'_{\pi_2^j} = \quad  
\frac2{n}\ \left|\begin{array}{ccccc} \ \  0 & \ \ s_{j} & \ \ \ s_{2j} & \ \ s_{3j} &\dots \\  \ \ -s_j & \ \ 0 & \ \ s_{j} & \ \ \ s_{2j} & \dots \\  \ \ \ -s_{2j} & \ \ -s_j & \ \ 0 & \ \  s_{j} & \dots \\  \dots & \dots & \dots & \dots &\dots  \end{array}\right|.$$
From the idempotent properties of $U_{\chi_j}$ and $U_{\pi_2^j}$, one has
$$U_{\pi_2^j} U'_{\pi_2^j} =  U'_{\pi_2^j}U_{\pi_2^j}  =  U'_{\pi_2^j},\qquad  (U'_{\pi_2^j})^2 = -U_{\pi_2^j}.$$
Furthermore, we have
$$P_C U_{\pi_2^j} P_C = U_{\pi_2^j},\qquad P_C U'_{\pi_2^j} P_C = -U'_{\pi_2^j}.$$
We immediately obtain
\begin{theorem}
The component $MM(D_{2n})_{\pi_2^j}$ associated to  $\pi_2^j$ in $MM(D_{2n})$ obtains a quaternionic structure as follows:  define a basis
$$q_1 = U_{\pi_2^j}, \quad  q_2 = iP_CU_{\pi_2^j}, \quad q_3 = U'_{\pi_2^j},\quad q_4 = iP_CU'_{\pi_2^j}.$$
Then, for all $1\le t \le 4,$ $q_1q_t=q_tq_1=q_t$, and 
$$q_2^2=q_3^2=q_4^2=-q_1,\quad q_2q_3q_4=-q_1.$$
\end{theorem}

\begin{example}  When $n=3,$ there is only one two-dimensional type, with corresponding basis in $MM(D_6)$, the space of semi-magic squares of size 3:
$$q_1 =
\frac1{3}\ \left|\begin{array}{ccc} \ \ \  2 & -1 &  -1  \\   -1& \ \ \ 2 &  -1  \\   -1 &  -1& \ \ \ 2  \end{array}\ \right|, \quad q_2 = \frac{i}{3}\ \left|\begin{array}{ccc}  -1 & -1 & \ \ \  2 \\   -1& \ \ \ 2 &  -1  \\  \ \ \ 2 &  -1& -1  \end{array}\ \right|,$$

$$q_3 = 
\frac{\sqrt{3}}{3}\ \left|\begin{array}{ccc}  \ \ \ 0 &  \ \ \  1 &  -1  \\   -1 & \ \ \  0 & \ \ \  1  \\ \ \ \    1 &  -1 & \ \ \  0  \end{array}\ \right|,\quad q_4 = 
\frac{i\sqrt{3}}{3}\ \left|\begin{array}{ccc}\ \ \  1 & -1 & \ \ \  0  \\   -1 & \ \ \  0 & \ \ \ 1  \\   \ \ \ 0 & \ \ \ 1 & -1 \end{array}\ \right|.$$
We obtain a basis for $MM(D_{6})$ by including $J$.

The case $n=4$ also has a single two-dimensional type. We have
$$q_1 =
\frac1{2}\ \left|\begin{array}{cccc} \ \ \ 1 & \ \ \ 0 &  -1 & \ \ \ 0 \\  \ \ \  0 &\ \ \  1 & \ \ \ 0 & -1  \\ -1 & \ \ \ 0 & \ \ \ 1 & \ \ \ 0  \\ \ \ \  0 & -1 &\ \ \  0 &\ \ \  1  \end{array}\ \right|,\quad q_2 = 
\frac{i}{2}\ \left|\begin{array}{cccc}  \ \ \  0 & -1 & \ \ \  0 & \ \ \ 1\ \\   -1 & \ \ \ 0 & \ \ \ 1 & \ \ \ 0\  \\ \ \ \ 0 & \ \ \  1 & \ \ \ 0 & -1\  \\  \ \ \ 1 & \ \ \ 0 &  -1 & \ \ \ 0 \ \end{array}\right|,$$  
$$q_3 = 
\frac1{2}\ \left|\begin{array}{cccc}  \ \ \  0 & \ \ \   1 & \ \ \   0 & -1\ \\   -1 & \ \ \  0 & \ \  \ 1 & \ \ \ 0\  \\ \ \ \   0 &  -1 & \ \ \   0 & \ \ \  1\  \\ \ \ \   1 & \ \ \ 0 & -1 & \ \ \  0  \end{array}\right|,\quad q_4 = \frac{i}{2}\ \left|\begin{array}{cccc}   \ \  \ 1 & \ \ \ 0 &  -1 & \ \ \ 0\ \\  \ \ \  0 & -1 & \ \ \ 0 & \ \ \ 1\  \\   -1 & \ \ \ 0 & \ \ \ 1 & \ \ \ 0\  \\ \ \ \  0 & \ \ \ 1 & \ \ \ 0 & -1 \ \end{array}\right|.$$
\end{example}
Again we obtain a basis for $MM(D_8)$ by including $J_1$ and $J_2.$

%
%
%
%

\end{document}